\documentclass[final,3p]{elsarticle}
\usepackage{lineno}
\modulolinenumbers[1]

\usepackage[utf8]{inputenc}
\usepackage{amsmath}
\usepackage{amsfonts}
\usepackage{amssymb}
\usepackage{amsfonts}
\usepackage{amsthm}
\usepackage{amscd}
\usepackage{color}
\usepackage{graphics}
\usepackage{graphicx}
\usepackage[toc,page]{appendix}
\usepackage{overpic}
\usepackage{multirow}
\usepackage[normalem]{ulem}

\usepackage[all]{xy}

\usepackage{algorithm,algorithmic}

\usepackage[T1]{fontenc}

\usepackage{tikz}

\newcommand{\C}{\mathbb{C}}
\newcommand{\R}{\mathbb{R}}

\newcommand{\Z}{\mathbb{Z}}

\newcommand{\A}{\mathbb{A}}

\newcommand{\I}{\mathrm{i}}

\newcommand{\f}[1]{\mathbf{#1}}

\newtheorem{theorem}{\bf Theorem}
\newtheorem{lemma}[theorem]{\bf Lemma}
\newtheorem{proposition}[theorem]{\bf Proposition}

\newtheorem{remark}[theorem]{\bf Remark}
\newtheorem{definition}[theorem]{\bf Definition}

\newtheorem{conjecture}[theorem]{\bf Conjecture}

\journal{arXiv}
\biboptions{authoryear}

\begin{document}

\sloppy

\begin{frontmatter}

\title{Characterization of polynomial surfaces of revolution and polynomial quadrics}

\author[plzen1,plzen2]{Michal Bizzarri}
\author[plzen1,plzen2]{Miroslav L\'avi\v{c}ka}
\author[madrid]{J. Rafael Sendra}
\author[plzen1,plzen2]{Jan Vr\v{s}ek}

\address[plzen1]{Department of Mathematics, Faculty of Applied Sciences, University of West Bohemia,
         Univerzitn\'i~8,~301~00~Plze\v{n},~Czech~Republic}

\address[plzen2]{NTIS -- New Technologies for the Information Society, Faculty of Applied Sciences, University of West Bohemia, Univerzitn\'i 8, 301 00 Plze\v{n}, Czech~Republic}

\address[madrid]{Department of Mathematics, CUNEF-University, Spain}

\begin{abstract}
In this paper, we characterize the polynomiality of surfaces of revolution by means of the polynomiality of an associated plane curve. In addition, if the surface of revolution is polynomial, we provide formulas for computing a  polynomial parametrization, over $\mathbb{C}$, of the surface. Furthermore, we perform the first steps towards the analysis of the existence, and actual computation, of real polynomial parametrizations of surfaces of revolution.  As a consequence, we give a complete picture of the real polynomiality of quadrics  and we formulate a conjecture for the general case.
 \end{abstract}


\end{frontmatter}

\section{Introduction}

In many applications, such as geometric modeling, parametric representations of curves and surfaces are more appropriate than implicit ones. Among these parametrizations, rational and polynomial parametrizations are particularly important. Let us focus on some advantages of polynomial parametrizations,
the most common cases of which are B\'ezier and B-spline forms. For instance, rational parametrizations can lead to numerical instability when the parameter values approach the poles of the rational functions. We could also point out the advantages of using polynomial parametrizations instead of rational ones when solving surface integrals. Also, polynomial parametrizations appears in the study of  globally identifiable parameters in   control theory (see \cite{FOS}).

The problem of computing a birational polynomial parameterization without base points of a rational surface (if it exists) was solved in \cite{SePD2020}.  Recently, an improvement that overcomes some shortcomings of the original algorithm has appeared in \cite{PDSeBeSh2023}. The great advantage of these approaches is that they deal with a wide class of rational surfaces. On the other hand, even if universal, or almost universal, methods exist, it is always worth studying special algorithms for certain classes of objects, as these often provide simpler solutions for the selected types of problems and helps to understand theoretically the corresponding family of objects.

For this reason, we decided to study polynomial parameterizations of surfaces of revolution, since these shapes are widely used in geometric modeling and computer graphics, as well as in a wide range of engineering and design applications.  Many interesting geometric objects, such as cylinders, cones, or tori, are surfaces of this type. In addition, surfaces of revolution are often used to create 3D digital surfaces for automotive and industrial design, we can find them in architecture, or in the manufacturing of various products. Let us also mention that the class of rational revolution  surfaces   are not included in the class of surfaces studied in \cite{SePD2020} and \cite{PDSeBeSh2023} and, hence, the results presented in this paper are relevant in the study of polynomial parametrizations of surfaces. 
 In \cite{VrLa2015} the problem of rationality and unirrationality of surfaces of revolution was thoroughly investigated. However, the paper was not devoted to the study of polynomial parameterizations. Thus, the identification of polynomial surfaces of revolution is still an open problem. This paper can therefore be seen as a continuation of \citep{VrLa2015} published in the \textit{Journal of computational and applied mathematics. }
 
A surface of revolution can be seen as a surface which is  invariant under the group of rotations along a fixed axis. Alternatively, one can visualize a surface of revolution as the surface obtained  by rotating a  curve  (the cross section of the surface with a plane containing the rotation axis)    
 about the axis. We call this rotating curve the profile curve, and we denote it by $P$.  In a first approximation, it seems natural to think that the relevant information about the rationality of the surface of revolution may reside in the profile curve. However, in \cite{VrLa2015}, the authors prove that  the rationality of the surface of revolution is characterized by means of the rationality of a different curve, denoted as $P^2$,  obtained through a quadratic transformation of the profile curve $P$ (see Fig. \ref{fig:strips} and Proposition \ref{prp:rational SOR}). In this paper, we show that $P^2$ also characterizes the polynomiality  of the surface of revolution. More precisely, the main contributions of this paper are
 \begin{itemize}
 \item We prove that the polynomiality (over $\mathbb{C}$) of a surface of revolution, that is not a cylinder of revolution, is characterized by means of the polynomiality  (over $\mathbb{C}$) of the curve $P^2$ (see Theorem  \ref{thm: S polynomial iff P2 polynomial}).
 \item We provide close formulas for computing a polynomial parametrization of a surface of revolution from a polynomial parametrization of $P^2$ (see Remark \ref{eq:final}).
 \item We take the first steps in the study of the existence, and actual computation, of real polynomial parametrizations of surfaces of revolution from  the existence of real polynomial parametrizations of the curve $P^2$ (see Section \ref{sec-real-param}). As a consequence, we give a complete analysis for the case of quadrics (see Section \ref{sec-quadrics}) and we formulate a conjecture for the general case (see Subsection \ref{subsec-conjecture}). 
 \end{itemize}

\section{Motivation and preliminary results}

It turns out that every surface parametrized by polynomials is an algebraic variety, and therefore the proper language for attacking the problem presented is that of algebraic geometry. Thus, we begin by introducing standard notions that will be used throughout the manuscript. Then we provide a recapitulation of known results about polynomial curves and the related results about surfaces of revolution. Throughout this paper we will work over the field $\mathbb{C}$ of the complex numbers, although the results are valid for any algebraically closed field of characteristic zero,   and eventually over the field $\mathbb{R}$ of the real numbers, in which case it will be conveniently indicated. Indeed, we will in general work with curves/surfaces as zero-sets of polynomials over $\mathbb{C}$. We will refer to real curves/surfaces when the zero-set contains infinitely many real points in the case of a curve or an Euclidean disc of real points in the surface case. 

\begin{definition}
A~\emph{polynomial} curve/surface $X$ is a curve/surface in the affine space $\A^n$ admitting a polynomial parametrization, i.e., a dominant map $\A^k\rightarrow X$, where $k=1$ or $k=2$, given by polynomial functions.
\end{definition}

The polynomial parametrizations are special instances of rational parametrizations -- those given by rational functions. Recall that a variety admitting a rational parametrization is called \emph{unirational}.  If there exists a parametrization which is additionally birational, then we speak of a \emph{rational} variety. According to L\"{u}roth's theorem, there is no distinction between rationality and unirrationality for curves. For surfaces, the field of definition comes into play.  In particular, over the field of complex numbers $\C$, every unirational surface must also be rational (see e.g. \cite{Sch13} page 208). However, this is no longer the case if we focus on real surfaces, which are known to be rational if and only if they are unirational and the desingularization of their projective closure is connected (see e.g. \cite{Roy1998} and \cite{Comessatti1913}).

We do not make a similar distinction for polynomial parametrizations. If the parametrization happens to be bijective on the Zariski open subset of $\A^k$, then there exists its inverse, which is generally given by rational functions.  Such a map is called a birational morphism, but we prefer to use the term \emph{proper parametrization} in this situation.

The polynomial curves are well known. They are characterized as rational curves with only one point at infinity. The rational parametrization of the curve provides an effective criterion for deciding the existence of a polynomial one. Let us summarize the relevant results from \cite[Section 6.2]{SeWiDi07} in the next proposition; the result is presented over $\mathbb{R}$, but it is also valid over $\mathbb{C}$.

\begin{proposition}\label{prp:polynomial curves}
Let $C$ be a (real) polynomial curve.  Then there exists a (real) polynomial proper parametrization. Furthermore, if $\f p(t)$ and $\f q(s)$ are two such parametrizations, then $\f q(s)=\f p(as+b)$.
\end{proposition}

The situation is much more complicated when we consider surfaces, due to the richness and complexity of Cremona transformations, i.e. birational transformations of the parametric domain. For example, two  proper polynomial parametrizations of the surface need not to be related by a polynomial reparametrization. Let $\phi:\A^2\rightarrow\A^2$ be given by $[s,t]=\phi(u,v)=[u^2-v^2,u+v]$ this is a~polynomial map which is almost everywhere bijective and thus it possesses an inversion. However
$$
[u,v] = \phi^{-1}(s,t)=\left[\frac{s+t^2}{2 t},\frac{-s+t^2}{2 t}\right],
$$ 
which is only rational. Therefore the reparametrization between two polynomial parametrizations might be rational. Moreover, it is not known which surfaces have polynomial parametrizations.  As an example, consider the following proposition from \cite{PeRo96}; in the sequel $\I$ denotes the imaginary unit.
 
\begin{proposition}\label{prp:S2 is poly}
The unit sphere $x^2+y^2+z^2=1$ can be parametrized polynomially as 
$$
[v-u(u v+1),2uv+1,\I(u(u v+1)+v)].
$$
\end{proposition}

As we can see, the parametrization is not real. In addition, it is easy to verify that the sphere cannot be parametrized by real polynomials because it is a~compact surface.

\smallskip
In the following, our goal is to characterize a class of polynomial surfaces of revolution.  Let us recall the definition.

\begin{definition}
A surface $S$ is a surface of revolution (SOR for short) if it is invariant under the group of rotations along a fixed axis.
\end{definition} 

Let $\f c(t)=[x(t),y(t),z(t)]$, where $z(t)$ is non-constant, be a parametric curve, and let the axis of rotation be the $z$-axis. Then the SOR obtained by rotating this curve about the axis can be given by
\begin{equation}\label{eq:rat param sor}
  \f s(s,t)=\left[x(t)\frac{2s}{1+s^2}+y(t)\frac{1-s^2}{1+s^2},-x(t)\frac{1-s^2}{1+s^2} +y(t)\frac{2s}{1+s^2} ,z(t)\right].
\end{equation}
If $\f c(t)$ was rational, then so is $\f s(s,t)$. The prominent curve on SOR is the so-called \emph{profile curve}, which is the section of the surface through a plane containing the axis. In the following we will always choose a coordinate system such that the axis of rotation coincides with the $z$-axis and the profile curve is the section through the plane $y=0$.  Therefore, we will usually treat it as a planar curve in the $xz$-plane.

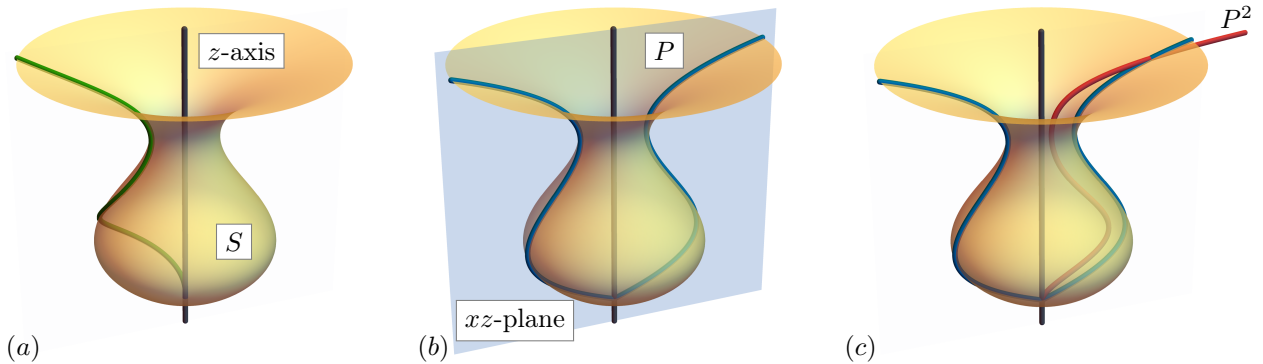
\begin{figure}[tbh]
\begin{center}
\begin{overpic}[height=0.28\columnwidth]{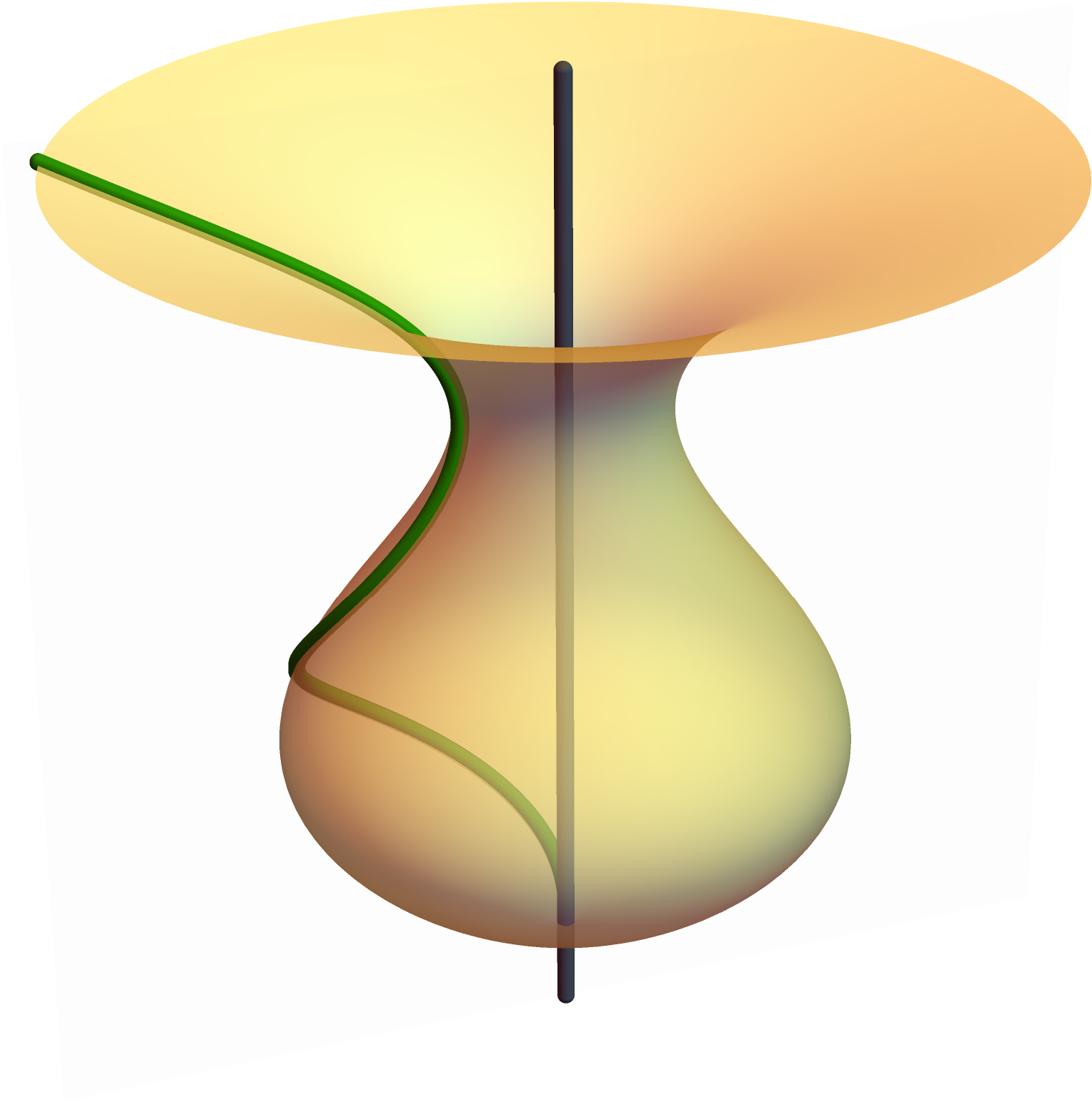}
\put(0,0){$(a)$}
\put(55,85){\fcolorbox{gray}{white}{$z$-axis}}
\put(60,30){\fcolorbox{gray}{white}{$S$}}
\end{overpic}
\hfill
\begin{overpic}[height=0.28\textwidth]{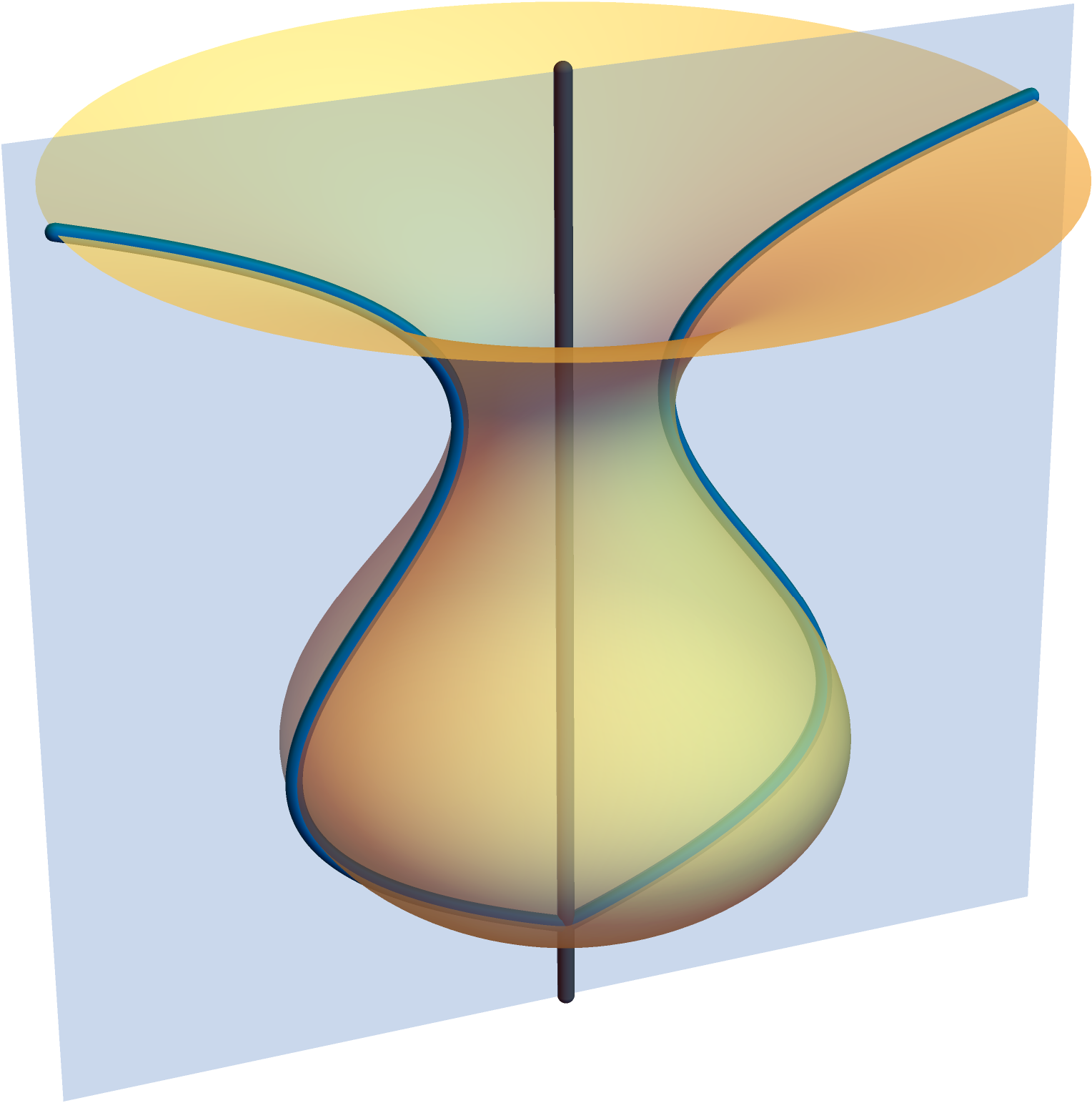}%,grid,tics=20
\put(-5,0){$(b)$}
\put(60,85){\fcolorbox{gray}{white}{$P$}}
\put(6,8){\fcolorbox{gray}{white}{$xz$-plane}}
\end{overpic}
\hfill
\begin{overpic}[height=0.28\textwidth]{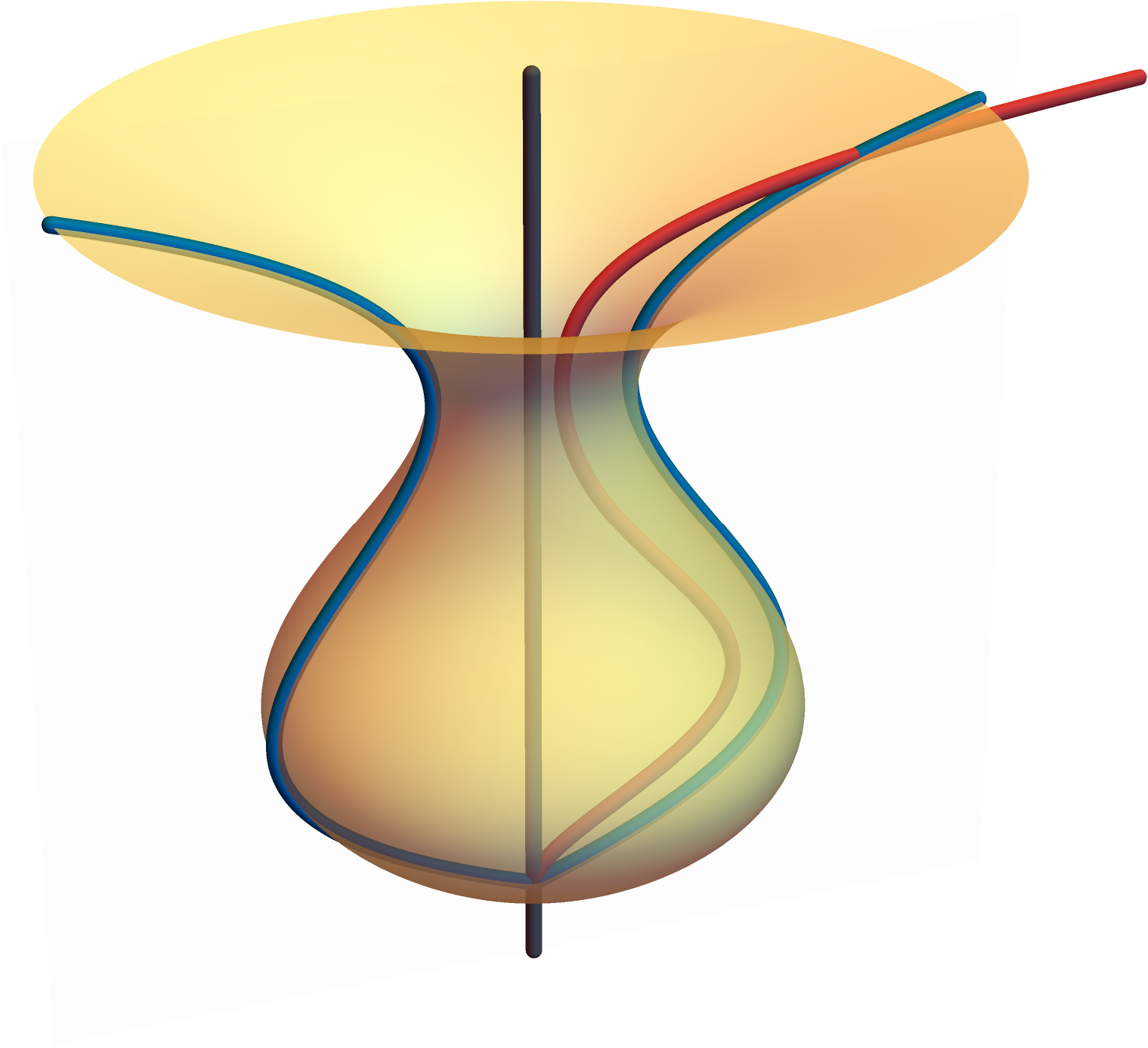}%,grid,tics=20
\put(-5,0){$(c)$}
%\put(92,73){\fcolorbox{gray}{white}{$P^2$}}
\put(93,86){$P^2$}
\end{overpic}
\caption{Surface of revolution.
(a)
Rotating a curve (green) around an $z$-axis yields a surface of revolution $S$.
(b)
Profile curve $P \subset S$ (blue) corresponds to the section of $S$ with the $xz$-plane.
(c)
Mapping $\A^2\rightarrow\A^2$ given by $[x,z]\mapsto[x^2,z]$ yields the curve $P^2 \not\subset S$ (red).
}\label{fig:strips}
\end{center}
\end{figure}

The rationality of the profile curve clearly implies the unirationality of the corresponding SOR. However, the reverse is not true and we will recall the result from \cite{VrLa2015}.   The profile curve of the surface $S$ is denoted by $P$. Then consider the mapping $\A^2\rightarrow\A^2$ given by $[x,z]\mapsto[x^2,z]$. Let $P^2$ be the image of the profile $P$ under this map. Note that since $P$ is symmetric about the $z$-axis, the points $[x,z]$ and $[-x,z]$ on $P$ are mapped to the same point on $P^2$, and thus the map $P\rightarrow P^2$ is a double cover. We get the following statement: 

\begin{proposition}\label{prp:rational SOR}
$S$ is a rational SOR (over $\C$) if and only if $P^2$ is a rational curve.
\end{proposition}

Since a circle is not a polynomial curve, the parametrization \eqref{eq:rat param sor} of the SOR respecting its rotational property cannot be polynomial even for a~polynomial profile curve.  On the other hand, there are polynomial SORs.  One of the simplest examples is a paraboloid of revolution $x^2+y^2-z=0$ with the parametrization $[u,v,u^2+v^2]$.   Another example is the sphere from Proposition~\ref{prp:S2 is poly}. Thus there is a promising chance that there is a whole non-trivial set of polynomial surfaces of revolution.  Their characterization is the main goal of the next section.

\section{Polynomial Surfaces of Revolution}

In this section we will characterize the polynomial surfaces of revolution.  We will work over the field of complex numbers $\C$, which will allow us to give a complete characterization.   Since we are aware of the importance of results over reals, these will be discussed in the next section. A first observation is that the first implication of Proposition~\ref{prp:rational SOR} can be immediately specialized to polynomial surfaces.

\begin{proposition}\label{thm: S polynomial then P2 polynomial}
If $S$ is a polynomial SOR (PSOR in short), then $P^2$ is a polynomial curve.
\end{proposition}

\begin{proof}
If $S$ is a PSOR, then there exists a polynomial curve $C\subset S$. Let $\f c(t)=[c_1(t),c_2(t),c_3(t)]$ be its parametrization. Obviously $c_3(t)$ is not constant, because a circle does not have a polynomial parametrization. This means that $[c_1^2(t)+c_2^2(t),c_3(t)]$ is a polynomial parametrization of the curve $P^2$ associated to~$S$.
\end{proof}

Now let us focus on the opposite implication of Proposition~\ref{thm: S polynomial then P2 polynomial}. Let a proper polynomial parametrization of $P^2$ be written in the form 
\begin{equation}\label{eq param of P2}
\left[p(t)a^2(t),b(t)\right],
\end{equation}
where $a,b,p\in \C[t]$ and $p$ is square-free. By Proposition~\ref{prp:polynomial curves} the parametrization is unique up to a linear reparametrization $\varphi(t)=\alpha t+ \beta$. In particular, the degree of $p$ is independent of the chosen parametrization and we will denote it by $\Delta_S$. In addition to any SOR $S$ with a polynomial $P^2$, we associate the so-called \emph{tubularization}, defined by the equation
\begin{equation}\label{eq:tubular}
  T_S:  x^2+y^2-p(z) = 0.
\end{equation}
Of course, this is only unique up to a linear change of the parameter in $p(z)$, which corresponds to an affine transformation of the form $[x,y,z]\mapsto[x,y,\alpha z+\beta]$. Note that in general a~\emph{tubular surface} is any surface of the form $A(z)x^2+B(z)y^2+C(z)$=0, see \cite{Sc00}. In a more general frame, a tubular surface is a pencil of conics and, as a consequence of Tsen's Theorem, it is a rational surface (see \cite[Corollaries 1.11 and 1.12]{Sch13}). The tubularization of SOR was first introduced in \cite{VrLa2015}, where it was used to study the rationality of surfaces of revolution.  It turns out that it can also be used effectively for polynomial SOR because of the following lemma.

\begin{lemma}\label{lem:tubularization}
Let $S$ be a SOR with $P^2$ parametrized by \eqref{eq param of P2} and $T_S$ be its tubularization.  Then there exists a polynomial map $\tau:T_S\rightarrow S$ given by 
\begin{equation}\label{eq:tubular map}
\tau:[x,y,z]\mapsto[ a(z)x,a(z)y,b(z)].
\end{equation}   
In particular, if $T_S$ is polynomial, then so is $S$. 
\end{lemma}

\begin{proof}
Consider a point $\f x_0=[x_0,y_0,z_0]\in T_S$.  Its image $\tau(\f x_0)$ satisfies 
$$(a(z_0)x_0)^2+(a(z_0)y_0)^2=a^2(z_0)(x_0^2+y_0^2)=a^2(z_0)p(z_0),$$
i.e., it is the rotated point $[a(z_0)\sqrt{p(z_0)},b(z_0)]$, which is actually a point on the profile curve and thus $\tau(\f x_0)\in S$. The determinant of the Jacobi matrix of the map $\tau$ is equal to $a^2(z)b'(z)$, which is generically nonzero whenever $b(z)$ is a non-constant polynomial.  This means that the map $\tau:T_S\rightarrow S$ defined by \eqref{eq:tubular map} is a dominant polynomial map whenever $S$ is not a plane perpendicular to the axis.
\end{proof}

\begin{remark}\rm
Note that the parametrization \eqref{eq param of P2} of $P^2$ also provides a radical parametrization  (see \cite{SENDRA20171} for further details)
\begin{equation}\label{eq:parm of S}
\f p(s,t) = \left[\frac{2s}{1+s^2}a(t)\sqrt{p(t)},\frac{1-s^2}{1+s^2}a(t)\sqrt{p(t)},b(t)\right].
\end{equation}
of $S$. We will use it later. Furthermore, 
\[  \left[\frac{2s}{1+s^2} \sqrt{p(t)},\frac{1-s^2}{1+s^2} \sqrt{p(t)},t\right] 
\]
is a radical parametrization of $T_S$.
\end{remark}

\begin{remark}\label{rem:st0}\rm
The converse of Lemma~\ref{lem:tubularization} is not true. For example, the cylinder of revolution $x^2+y^2-1=0$ is not a~polynomial surface.   
In fact, every polynomial surface contains a~rich family of polynomial curves -- the images of polynomial curves in the parameter domain $\A^2$ under the parametrization.  However, the only polynomial curves on the cylinder are its rulings. If $[x(t),y(t),z(t)]$ would be a polynomial curve on the cylinder, which is not a ruling, then its projection $[x(t),y(t)]$ on the first two coordinates would be a~polynomial parametrization of the circle. This is a contradiction. 

On the other hand, any SOR $S$ whose tubularization is the cylinder of revolution is polynomial.  To see this, consider a parametrization \eqref{eq param of P2} of $P^2$.  Now $p(t)$ is constant, which means that the profile curve $P$ consists of two components parametrized by $[\pm a(t),b(t)]$; we assume w.l.o.g. that $p(t)=1$. Furthermore, the surface $S$ is a cylinder if and only if $a(t)$ is constant. Assume (possibly after some suitable reparametrization) that 0 is a root of $a(t)$ and consider the substitution 
$$
[s,t] \mapsto \left[-\frac{u}{v}, u^2 + v^2\right].
$$
This transforms \eqref{eq:parm of S} into
$$
\left[-\frac{2 u v  \, a\left(u^2+v^2\right)}{u^2+v^2},\frac{\left(v^2-u^2\right) a\left(u^2+v^2\right)}{u^2+v^2},b\left(u^2+v^2\right)\right].
$$
Since we assumed $a(t)=t\tilde{a}(t)$, we get a polynomial parametrization of the surface $S$
\begin{equation}\label{eq:param case 0}
\left[-2uv\tilde{a}\left(u^2+v^2\right),\left(v^2-u^2\right) \tilde{a}\left(u^2+v^2\right),b\left(u^2+v^2\right)\right].
\end{equation}
\end{remark}

Let us further consider situations with non-constant $p(t)$, i.e., we now assume that $\Delta_S>0$. 
First, note that the tubular surface $T_S:x^2+y^2-p(z)=0$ can be parametrized rationally as
$$
\f q(s,t)=\left[\frac{\mathrm{i}}{2s}(s^2-p(t)),\frac{1}{2s}(s^2+p(t)) ,t \right].
$$
Clearly, $\f q(s,t)$ fulfills the defining equation of $T_S$ and the Jacobian of the map is of rank 2 generically, i.e., it is a parametrization of the surface. Using the fact that $p$ is not constant, let $\alpha$ be a  root of $p$ and consider a~birational map

$$
\begin{array}{ccccc}
\Phi: & \mathbb{C}^2 &  \rightarrow & \mathbb{C}^2 & \\
      & [s,t] & \mapsto &  [v, u v+\alpha]. 
\end{array} 
$$

Since $p(u\cdot 0+\alpha)=p(\alpha)=0$  we can factor out the variable $v$ from this expression, i.e., $p(u v+\alpha)=v h(u,v)$ for some polynomial $h$. Hence we arrived at
\begin{equation}\label{eq:param of tubular}
\f q(\Phi(u,v)) = \left[ \frac{\mathrm{i}}{2} (v-h(u,v)),\frac{\mathrm{1}}{2} (v+h(u,v)), u v + \alpha  \right],
\end{equation}
which is a polynomial parametrization of $T_S$. Therefore, the original SOR $S$ is polynomial by Lemma~\ref{lem:tubularization} and we just proved the following theorem.

\begin{theorem}\label{thm: S polynomial iff P2 polynomial}
$S$ is PSOR (over $\C$) if and only if it is not a cylinder of revolution and $P^2$ is a polynomial curve.
\end{theorem}

\begin{remark}\label{eq:final}
Combining \eqref{eq:param of tubular} with Lemma~\ref{lem:tubularization}, the parametrization of the original SOR $S$ with $P^2$ parametrized by \eqref{eq param of P2} can be given explicitly as
$$
\f s(u,v) = \left[ \frac{\mathrm{i}}{2} a(u v + \alpha) (v-h(u,v)),\frac{\mathrm{1}}{2} a(u v + \alpha) (v+h(u,v)), b(u v + \alpha)  \right].
$$
\end{remark}

\section{Polynomial Surfaces of Revolution with Real Parameterizations}\label{sec-real-param}

Real parametrizations play a crucial role in applications. Therefore, in this section we will focus on real SORs. Of course, the curve $P^2$ is considered to be polynomial and its parametrization \eqref{eq param of P2} can again be written using real polynomials $a,b,p$. We will give a partial answer to the question of the existence of a real parametrization of the surface depending on the degree $\Delta_S$.

\subsection*{$\boldsymbol{\Delta_S=0}$\bf\,:}
This case is solved in Remark~\ref{rem:st0}. The (complex) parametrization exists if $a(t)$ is not constant. The resulting formula \eqref{eq:param case 0} assumes that $0$ is the root of $a(t)$. If $a(t)$ has a real root, then this can always be achieved by a real reparametrization and the final parametrization of the surface is real as well. 
%\sout{However, if $a$ does not have a real root, then we do not know whether there exists a real parametrization or not.}
%
Next, let us consider $a(t)$ without real roots. Then $P$ is (possibly after a suitable reparametrization) of the form
$$
\left[(t^2+1)\hat{a}(t),b(t)\right],
$$
where $\hat{a}(t)$ has only imaginary roots or it is a constant. The corresponding parametrization \eqref{eq:parm of S} is
$$
\f p(s,t) = \left[\frac{2s}{1+s^2}(t^2+1)\hat{a}(t),\frac{1-s^2}{1+s^2}(t^2+1)\hat{a}(t),b(t)\right].
$$
We perform a transformation  of the form $s=A/B$, $t=C$  where $A,B,C$ are polynomials in $u,v$ and we get
\begin{equation}\label{eq:Delta0_norealroot}
\left[\frac{2AB}{A^2+B^2}(C^2+1)\hat{a}(C),\frac{B^2-A^2}{A^2+B^2}(C^2+1)\hat{a}(C),b(C)\right].
\end{equation}
For example, if we choose the polynomials $A,B,C$ such that $C^2+1=A^2+B^2$, then \eqref{eq:Delta0_norealroot} is polynomial. In the part dedicated to $\Delta_S=2$ we will show that such polynomials exist and thus also for $a(t)$ without real roots we can find a real polynomial parametrization of SOR.

\subsection*{$\boldsymbol{\Delta_S=1}$\bf\,:}
In this case $p$ is a linear polynomial, so the parametrization of $P^2$ can be written in the form $[t a^2(t),b(t)]$, i.e., $p(t)=t$ and $a(t),b(t)\in\R[t]$. The corresponding tubular surface is in fact a paraboloid of revolution with real polynomial parametrization $[u,v,u^2+v^2]$, which gives a real polynomial parametrization of the surface $S$:
$$
\f s(u,v) = \left[ a(u^2+v^2)u,a(u^2+v^2)v,b(u^2+v^2)\right].
$$

\subsection*{$\boldsymbol{\Delta_S=2}$\bf\,:}
Now the tubular surface \eqref{eq:tubular} is a quadric of revolution. Applying a real linear reparametrization, we can choose the representative with $p(z)=\pm z^2+\lambda$, where $\lambda\in \R$ is nonzero. Depending on the sign of $\lambda$, the situation splits into four subcases that must be treated separately:

%----------EMPTY SET-----------------------------------------------------
\paragraph{$\boldsymbol{p(z)=-z^2-|\lambda|}$}  The surface $S$ has no real point and therefore does not have a real parameterization.

%--------SPHERE----------------------------------------------------------
\paragraph{$\boldsymbol{p(z)=-z^2+|\lambda|}$}  In this case $T_S$ is a sphere. As already mentioned bellow Proposition~\ref{prp:S2 is poly} it does not posses any real parameterization as well.

%------HYPERBOLOID OF 1 SHEET----------------------------------------------

\paragraph{$\boldsymbol{p(z)=z^2+|\lambda|}$} The surface $T_S$ is a one-sheeted hyperboloid.  We can use the parametrization from Proposition~\ref{prp:S2 is poly} to give its real proper polynomial parametrization in the form
\begin{equation}\label{eq:param_1_sheeted_hyp}
   \sqrt{|\lambda|} [v-u (u v+1),2 u v+1,u (u v+1)+v].
\end{equation}
Combining this with formula \eqref{eq:tubular map} from Lemma~\ref{lem:tubularization}, we get a real proper polynomial parametrization of any SOR $S$, whose associated curve $P^2$ can be parametrized as $[(t^2+|\lambda|)a^2(t),b(t)]$, in the following form:
$$
\left[\sqrt{|\lambda|} a(\sqrt{|\lambda|}(u (u v+1)+v))(v-u (u v+1)),\sqrt{|\lambda|} a(\sqrt{|\lambda|}(u (u v+1)+v))(2 u v+1),b(\sqrt{|\lambda|} ((u v+1)+v)) \right].
$$
Moreover, going back to the situation $\Delta_S=0$, it is parametrization \eqref{eq:param_1_sheeted_hyp}  of the one-sheeted hyperboloid $A^2+B^2-C^2=1$ that guarantees the polynomiality of the real parametrization \eqref{eq:Delta0_norealroot}. Hence, it is enough to take  $A=v-u (u v+1)$, $B=2 u v+1$, $C=u (u v+1)+v$ and substitute it to \eqref{eq:Delta0_norealroot}.
%
%
%gives the choice $A=v-u (u v+1)$, $B=2 u v+1$, $C=u (u v+1)+v$ guaranteeing the polynomiality of \eqref{eq:Delta0_norealroot} and hence yielding a sought real polynomial parametrization when $a(t)$ has no real roots.
%$$
%\left[2(-u + v - 3 u^2 v + 2 u v^2 - 2 u^3 v^2)\hat{a}(u^2v+u+v),(1 - u^2 + 6 u v - 2 u^3 v - v^2 + 6 u^2 v^2 - u^4 v^2)\hat{a}(C),b(u^2v+u+v)\right].
%$$

%-----HYPERBOLOID OF 2 SHEETS--------------------------------------------------
\paragraph{$\boldsymbol{p(z)=z^2-|\lambda|}$} Finally, $T_S$ is a two-sheeted hyperboloid.  If $\varphi: \A_\R^2\rightarrow\A^3_\R$ is a polynomial map, then the image must be connected.   This means that $T_S$ cannot have a proper real polynomial parametrization. Nevertheless we will show that there still exists a polynomial parametrization doubly covering one sheet of the surface. From the theory of  Diophantine equation it is known, cf. \cite{rose1995},   how to find integer solutions  of the equation $x^2+y^2-z^2=-1$.  In particular, it holds
\begin{equation}\label{Dioph}
	(q_1 q_2 + q_3 q_4)^2+
	\left(\frac{q_1^2 + q_3^2 - q_2^2 - q_4^2}{2}\right)^2-
	\left(\frac{q_1^2 + q_3^2 + q_2^2 + q_4^2}{2}\right)^2=
	-(q_1 q_4 - q_2 q_3)^2.
\end{equation} 
Hence, any integer $q_1,q_2,q_3,q_4\in\Z$ satisfying the constraint $q_1 q_4 - q_2 q_3=1$ provides an integer point on the two-sheeted hyperboloid. Clearly, we may use this observation to solve our parametrization problem as well. All we have to do is to~find an arbitrary polynomial 2-surface lying on the hypersurface $q_1 q_4 - q_2 q_3=1$. Notice that its section by hyperplane $q_1=q_4$ is in fact  the hyperboloid of one-sheet  $q_1^2-q_2q_3=1$.  Using the result from the previous paragraph  we can parametrize it
$$
\left[\begin{array}{c}
	v - u (1 + u v) \\[0.5ex]
	1 + v + 2 u v + u (1 + u v) \\[0.5ex]
	-1 + v - 2 u v + u (1 + u v)
\end{array}
\right]^{\top}.
$$
Applying this parametrization and formula \eqref{Dioph} yields a real polynomial  parametrization of the two-sheeted hyperboloid in the form
\begin{equation}\label{eq:two sheet param}
\f q(u,v)=
\left[\begin{array}{c}
-2 (u^2 + 2 u^3 v - v^2 + u^4 v^2)\\[0.5ex]
-2 (u + v + 3 u^2 v + 2 u v^2 + 2 u^3 v^2) \\[0.5ex]
1 + 4 u v + 4 u^3 v + 2 v^2 + 2 u^4 v^2 + u^2 (2 + 4 v^2) 
\end{array}
\right]^{\top}.
\end{equation}
After the scaling $\sqrt{|\lambda|}\f q(u,v)$ we arrive at the real polynomial parametrization of the tubular surface $T_S$. The parametrization of the original surface $S$ is then obtained again by composing this with \eqref{eq:tubular map}. It is easy to see that \eqref{eq:two sheet param} is non-proper. In fact it doubly covers the ``upper'' sheet of the hyperboloid.

\subsection*{$\boldsymbol{\Delta_S=3}$\bf\,:}
Although we have not been able to fully characterize surfaces with real parametrizations, we provide a particular example of a surface to support our final conjecture. Namely, we discovered the following real polynomial parametrization of the cubic surface $x^2+y^2-z^3-1=0$  
$$
\left[\begin{array}{c}
 \dfrac{1}{2} u^3 \left(v^3-2 \sqrt{3} v^2+4 v-\sqrt{3}\right)-\dfrac{u^2}{2}+1 \\[0.5ex]
 \dfrac{1}{2} u^3 \left(\sqrt{3} v^3-4 v^2+2 \sqrt{3} v-1\right)+\dfrac{1}{2} u^2 \left(2 \sqrt{3} v^2-4 v+\sqrt{3}\right)+u \\[0.5ex]
 u^2 \left(v^2-\sqrt{3} v+1\right)+u v
\end{array}
 \right]^{\top}.
$$

\subsection*{\bf The conjecture}\label{subsec-conjecture}

Since the surface possessing a proper real polynomial parametrization must be connected and non-compact, and the real components of $T_S$ correspond to the intervals where $p(z)>0$, we immediately see that the necessary condition on the existence of such a parametrization is that $p(z)$ either has no real root and is positive, or it has exactly one real root. (Note that if $p(z)$ is negative everywhere, then the real part of the surface $T_S$ is empty). We have shown for $\Delta_S =1,2$ that the converse is also true, and we have provided a supporting example for $\Delta_S=3$.  At the same time, we believe that this holds not only for a tubular surface $T_S$, but also for the original surface $S$. Therefore we formulate the following conjecture, the verification of which is still an open problem.

\begin{conjecture}
PSOR $S$ has a proper real polynomial parametrization if and only if it is two-dimensional as a real set and the associated polynomial $p(t)$ has at most one real root.
\end{conjecture}

\section{Addendum -- (real) polynomial parametrizations of  quadrics}\label{sec-quadrics}
We have seen in the previous sections that the sphere, the hyperboloids of revolution, the cone of revolution, and the paraboloid of revolution possess polynomial parametrizations (a non-real one in the case of the sphere).  Clearly, this does not specialize to quadrics of revolution, since the polynomiality is preserved under affine transformations.  The hyperbolic paraboloid is also a polynomial surface, and we can conclude that all irreducible quadrics are polynomial surfaces except the cylinders over an~ellipse and a~hyperbola. This is summarized in the following table.

\begin{table}[H]
\begin{center}
\begin{tabular}{|ccc|}
\hline
\multicolumn{3}{|c|}{Existence of polynomial parametrizations of real irreducible quadrics}                                                                                      \\ \hline
\multicolumn{1}{|c|}{\multirow{4}{*}{\begin{tabular}[c]{@{}c@{}}singular\\ quadrics\end{tabular}}} & \multicolumn{1}{c|}{cone}                      & yes                        \\ \cline{2-3} 
\multicolumn{1}{|c|}{}                                                                             & \multicolumn{1}{c|}{elliptic cylinder}         & no                         \\ \cline{2-3} 
\multicolumn{1}{|c|}{}                                                                             & \multicolumn{1}{c|}{hyperbolic cylinder}       & no                         \\ \cline{2-3} 
\multicolumn{1}{|c|}{}                                                                             & \multicolumn{1}{c|}{parabolic cylinder}        & yes                        \\ \hline
\multicolumn{1}{|c|}{\multirow{5}{*}{\begin{tabular}[c]{@{}c@{}}regular\\ quadrics\end{tabular}}}  & \multicolumn{1}{c|}{ellipsoid}                 & yes (over $\C$ only)            \\ \cline{2-3} 
\multicolumn{1}{|c|}{}                                                                             & \multicolumn{1}{c|}{hyperboloid of one sheet}  & yes                        \\ \cline{2-3} 
\multicolumn{1}{|c|}{}                                                                             & \multicolumn{1}{c|}{hyperboloid of two sheets} & yes (over $\C$ or non-proper over $\R$) \\ \cline{2-3} 
\multicolumn{1}{|c|}{}                                                                             & \multicolumn{1}{c|}{hyperbolic paraboloid}     & yes                        \\ \cline{2-3} 
\multicolumn{1}{|c|}{}                                                                             & \multicolumn{1}{c|}{elliptic paraboloid}       & yes                        \\ \hline
\end{tabular}
\end{center}
\end{table}

\section{Summary}

Building on the previous work \citep{VrLa2015} devoted to rational surfaces of revolution, we studied the polynomiality of these shapes. It was shown that the polynomiality is related to a certain associated plane curve and its properties. We presented a method for computing a polynomial parametrization of the surface over the complex numbers. The use of so-called tubular surfaces, which always represent a class of surfaces of revolution, proved to be a suitable tool. In addition, real polynomial parametrizations for surfaces of revolution have been studied, giving insights into their existence and computation. This investigation led to a conjecture concerning the general case, which was supported by several particular results. Finally, our study led to a comprehensive understanding of the real polynomiality of quadrics, which is also a result that we are not aware of being published anywhere.

%%%%%%%%%%%%%%%%%%%%%%%%%%%%%%%%%%%%%%%%%%%%%%%%%%%%%%%%%%%%%%%%%%%%%%%%%%%%%%%%%%%%%%%%%%%%%%%%%%%%%%%%%%%%%%%%%%%%%%%%%%%%%%%%%%%%%%%%
\section*{Acknowledgments}
%%%%%%%%%%%%%%%%%%%%%%%%%%%%%%%%%%%%%%%%%%%%%%%%%%%%%%%%%%%%%%%%%%%%%%%%%%%%%%%%%%%%%%%%%%%%%%%%%%%%%%%%%%%%%%%%%%%%%%%%%%%%%%%%%%%%%%%%
%

\end{document}